\documentclass[12pt,a4paper,final]{article}

\usepackage{amscd}
\usepackage{amsfonts}
\usepackage{amsmath}
\usepackage{amssymb}
\usepackage{amstext}
\usepackage{amsthm}
\usepackage{color}
\usepackage[nodate]{datetime}
\usepackage{dsfont}
\usepackage{enumerate}
\usepackage{fancyhdr}
\usepackage[hmargin=25mm,top=25mm,bottom=30mm,paperwidth=210mm,paperheight=11in]{geometry}
\usepackage{graphicx}
\usepackage[latin1]{inputenc}
\usepackage{mathrsfs}
\usepackage{pifont}
\usepackage{psfrag}
\usepackage[normalem]{ulem}
\usepackage{url}
\usepackage{verbatim}
\usepackage{wasysym}

\usepackage[active]{srcltx}
\DeclareMathOperator*{\argmin}{arg\,min}

\newcommand{\oo}{\infty}

\newcommand{\F}{\mathcal{F}}

\renewcommand{\P}{\mathcal{P}}
\newcommand{\R}{\mathcal{R}}

\newcommand{\X}{\mathcal{X}}
\newcommand{\Y}{\mathcal{Y}}
\newcommand{\Z}{\mathcal{Z}}

\newcommand{\EE}{\mathbb{E}}
\newcommand{\NN}{\mathbb{N}}
\newcommand{\PP}{\mathbb{P}}
\newcommand{\RR}{\mathbb{R}}
\newcommand{\ZZ}{\mathbb{Z}}

\newcommand{\Trm}{T_{\RR_-}}
\newcommand{\Txm}{T_{X_{-1}}}

\newcommand{\dd}{{\mathrm d}}

\newcommand{\I}{\mathds{1}}

\newcommand{\equivalent}[1]{{\underset{#1}{\sim}}}

\theoremstyle{plain}
\newtheorem{theorem}{Theorem}

\newtheorem{lemma}[theorem]{Lemma}
\newtheorem{proposition}[theorem]{Proposition}
\newtheorem{claim}[theorem]{Claim}
\newtheorem*{assumption*}{Assumption}
\newtheorem*{lemma*}{Lemma}
\newtheorem*{proposition*}{Proposition}
\newtheorem*{theorem*}{Theorem}
\theoremstyle{definition}
\newtheorem*{definition*}{Definition}
\theoremstyle{remark}

\newtheorem*{remark*}{Remark}

\date{}

\renewcommand{\geq}{\geqslant}
\renewcommand{\leq}{\leqslant}

\renewcommand{\baselinestretch}{1.15}
\definecolor{orange}{cmyk}{0,0,0.2,0}

\begin{document}

\title{
Greedy Clearing of Persistent Poissonian Dust
}

\author{
L. T. Rolla, V. Sidoravicius, L. Tournier
\\ {\small Instituto de Matem\'atica Pura e Aplicada}
\\ {\small Universit\'e Paris XIII}
}

\maketitle
\begin{abstract}
Given a Poisson point process on $\RR$, assign either one or two marks to each point of this process, independently of the others. 
We study the motion of a particle that jumps deterministically from its current location to the nearest point of the Poisson point process which still contains at least one mark, and removes one mark per each visit.
A point of the Poisson point process which is left with no marks is removed from the system.
We prove that the presence of any positive density of double marks leads to the eventual removal of every Poissonian point.
\end{abstract}

This preprint has the same numbering of sections, figures, equations and theorems as the the
published article
``\emph{Stochastic Process. Appl. 124 (2014), 3496-3506.}''

\section*{Introduction}

Given a Poisson point process  $\P \subseteq \RR$, we assign a random number~$N_x \in \{1,2,3,\dots\}$ of marks to each point $x \in \P$, independently from the values assigned to the other points in~$\P$.
Consider the following discrete motion of a particle.
It starts at the origin~$\mathbf{0}$, jumps to the nearest point $x\in \P$, and removes one mark from~$x$.
Then it carries this procedure indefinitely: always removing one mark at its current position, and targeting for its next jump the nearest~$y\in\P$ still containing a mark.

If $\PP(N_x=1)=1$, an elementary application of Borel-Cantelli lemma shows that this motion eventually chooses a random direction, and drifts away, leaving a half-line unvisited.
In this work we show that for $\PP(N_x=1)<1$ and
$N_x$ taking values on~$\{1,2\}$,
every mark and, as a consequence, every Poissonian point is eventually removed from~$\RR$.
We conjecture that the same is true assuming only~$\PP(N_x=1)<1$.

{\textit{Few words on the motivation and the background.}}
A greedy algorithm reflects the strategy of maximizing the performance on the short run.
Suppose that the task is to visit a given set of points within an infinite region, and the strategy is always to choose the nearest non-visited point.
The model  was introduced in~\cite{lima-martinez-kinouchi-01}, where it was called  the ``local traveling salesman problem'' or the ``tourist walk''.\footnotemark{}
A tourist wishes to pay a visit to every landmark, and always goes to the nearest non-visited one.
Does this strategy succeed?
The answer may depend on the dimension.

If the landmarks form a Poisson point process on $\RR$, the answer is simple.
Consider the region spanning between the leftmost non-visited Poisson point on the right of the walk  and the rightmost non-visited Poisson point on its left.  The length of this region increases with each step of the walk, and a standard application of Borel-Cantelli Lemma implies that the walk crosses it only finitely many times.
As a consequence, the walk eventually begins to move monotonically either to $+\infty$ or to $-\infty$, thus leaving infinitely many Poisson points not visited.

On~$\RR^d$, $d\geq 2$, 
it is  less clear what one should expect from the behavior of the walk. This is due to certain self-repelling mechanism which lies in the nature of the process.
\footnotetext
{
The model is reminiscent of walks in rugged landscapes or zero-temperature spin-glass dynamics.
It was advertised in~\cite{stanley-buldyrev-01} and studied in~\cite{santos-et-al-07} and~\cite{boyer-08}, and further discussed in~\cite{bordenave-foss-last-11}.
}%
To make this more explicit, we introduce the following \emph{explorer}\footnote
{
We call this process the explorer process due to its close analogy with the rancher process, introduced and studied in~\cite{angel-benjamini-virag-03}, and also in~\cite{zerner-05}.
}
process.
The explorer
starts at $S_0 = \mathbf{0}$. Sample an exponentially-distributed random variable~$A_1$, which is interpreted as the volume that the particle is capable to explore at the first step.
Let~$D_1\subseteq \RR^2$ be the unique ball centered at $S_0$ and whose volume satisfies $|D_1|=A_1$.
The new position~$S_1$ of the explorer is then sampled uniformly on $\partial D_1$.
For the second step, we sample a new exponentially-distributed random variable~$A_2$, and let $D_2$ be the unique ball centered at~$S_1$ such that $|D_2 \setminus D_1|=A_2$.
The position~$S_2$ is then sampled uniformly on $(\partial D_2) \setminus D_1$.
\begin{figure}[h!]
\centering{
\includegraphics[width=5cm]{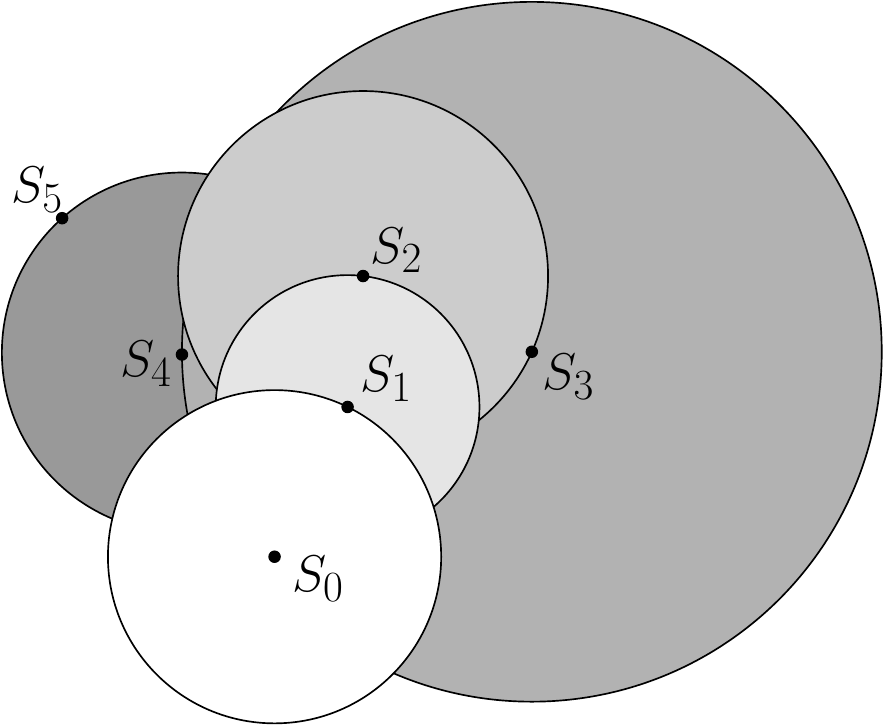}
}
\caption{\small First five steps of the explorer.} 
\label{fig:chicken}
\end{figure}
In general, given positions $S_0,\dots,S_{k-1}$, we sample a new exponentially-distributed random variable~$A_k$, and let $D_k$ be the unique ball centered at $S_{k-1}$ such that
\begin{equation}\label{1PD}
\left| \big. D_k \setminus (D_1 \cup \cdots \cup D_{k-1}) \right| = A_k,
\end{equation} 
and sample $S_k$ uniformly on 
\begin{equation}\label{2PD}
(\partial D_k) \setminus (D_1 \cup \cdots \cup D_{k-1}).
\end{equation}
The sequence~$S_n$ has the same law as the path performed by the tourist walk.
Figure~\ref{fig:chicken} illustrates the first five steps of the explorer.

Note that if requirement~\eqref{2PD} is replaced by a choice of $S_k$ uniformly on $\partial D_k$, then one may view the walk as a Brownian motion observed at random times given by~\eqref{1PD}.
Recurrence of the Brownian motion in the plane implies that every point is eventually covered by a disc $D_k$, while transience in higher dimensions implies that some regions are never explored.
Requirement~(\ref{2PD}) above forces new positions of the walk to be chosen away from the previously explored region, which is responsible for a \emph{local self-repulsion} of the process.
In the Euclidean space $\RR^d$ with $d \geqslant 3$, one naturally expects from the above description that some landmarks remain unvisited forever.
In the plane however, the question is more delicate.
Figure~\ref{fig:simul1} displays simulations of the process.

\begin{figure}[b!]
\centering{
\includegraphics[width=79mm]{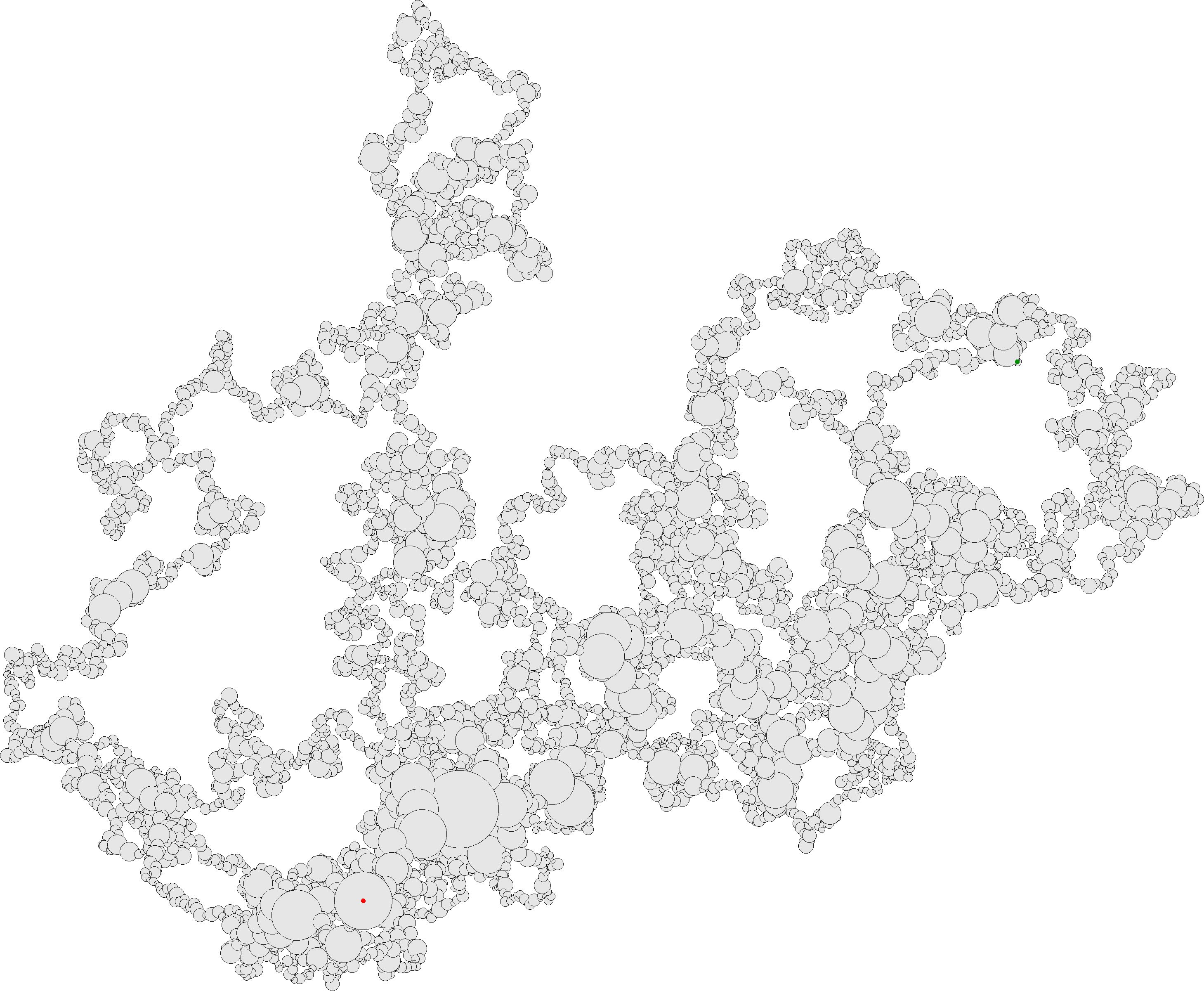}
\hfill
\includegraphics[width=79mm]{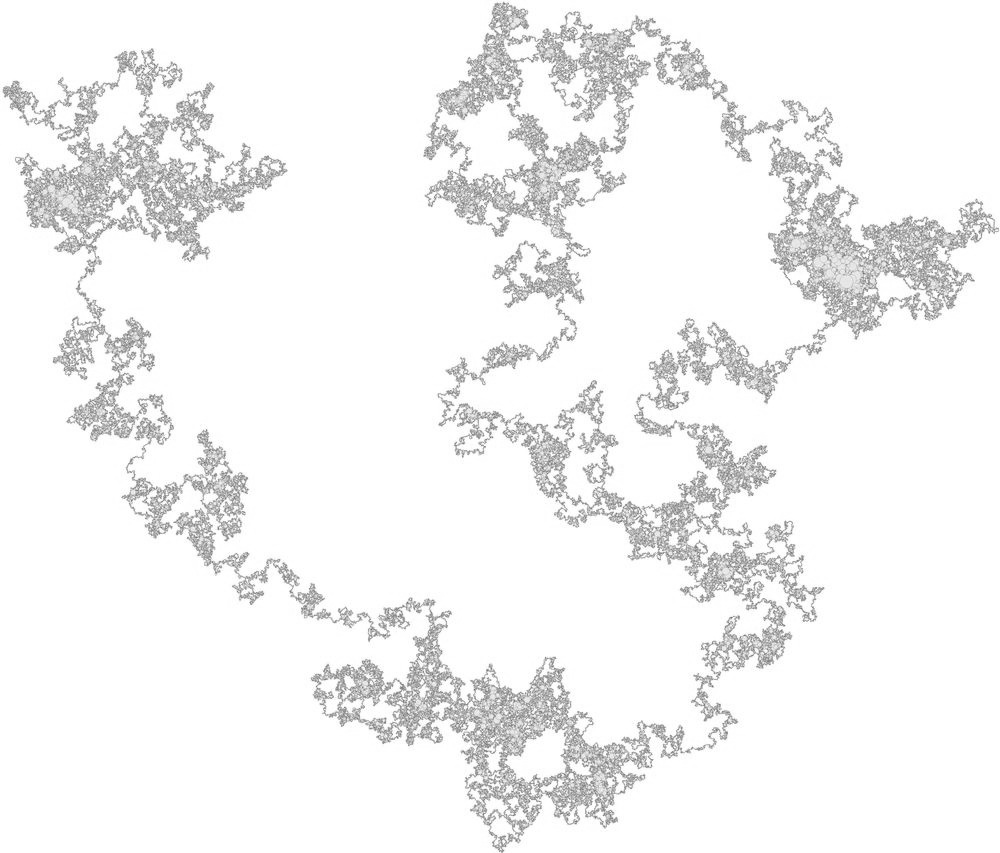}
}
\caption{\small Left: Simulation of $10^4$ steps of a greedy walk in $\RR^2$.
Discs refer to the description as the explorer process.
They are however drawn in inverse order (\textsl{i.e.}, chronologically) so as to give a clearer view of the largest discs, corresponding to largest jumps.
Right: Similar simulation with $10^6$ steps.}
\label{fig:simul1}
\end{figure}

A surprising case is that of a strip~$\RR\times[0,\varepsilon]$, which turns out to be drastically different from the line.
In numerical studies, a very intricate behavior, including heavy-tailed backtrack lengths, was noticed for strip-like rectangles~\cite{santos-et-al-07} and then in the infinite strip~\cite{boyer-08}.
This observation even suggests that any point will eventually be visited.
Such change in the behavior is caused by the following fact.
In the one-dimensional case, if particle enters the segment $[x_1, x_2]$ at position $x_1$ and leaves it at a position $x_2 > x_1$, then every Poisson point lying in $[x_1, x_2]$ is removed.
However on the strip of width $\epsilon > 0$, if the particle enters $[x_1,x_2]\times[0,\epsilon]$ at some position $(x_1,y_1)$ and leaves at some position $(x_2,y_2)$, it still may leave unvisited Poisson points in $[x_1,x_2]\times[0,\epsilon]$.
These Poissonian points that are left behind play a crucial role in reverting the direction of the walk.

In this paper we consider a variation of the model on the line $\RR$, where each point may independently, with probability $p>0$, contain two marks instead of one, so as to mimic the possibility for the walk on~$\RR\times[0,\varepsilon]$ to leave a point behind, which we think of as a persistent grain of dust.
We prove that with this modification all marks are eventually removed, for any value of~$p$.
See Figure~\ref{fig:simul3}.

\begin{figure}[b!]
\centering{
\includegraphics[width=9cm]{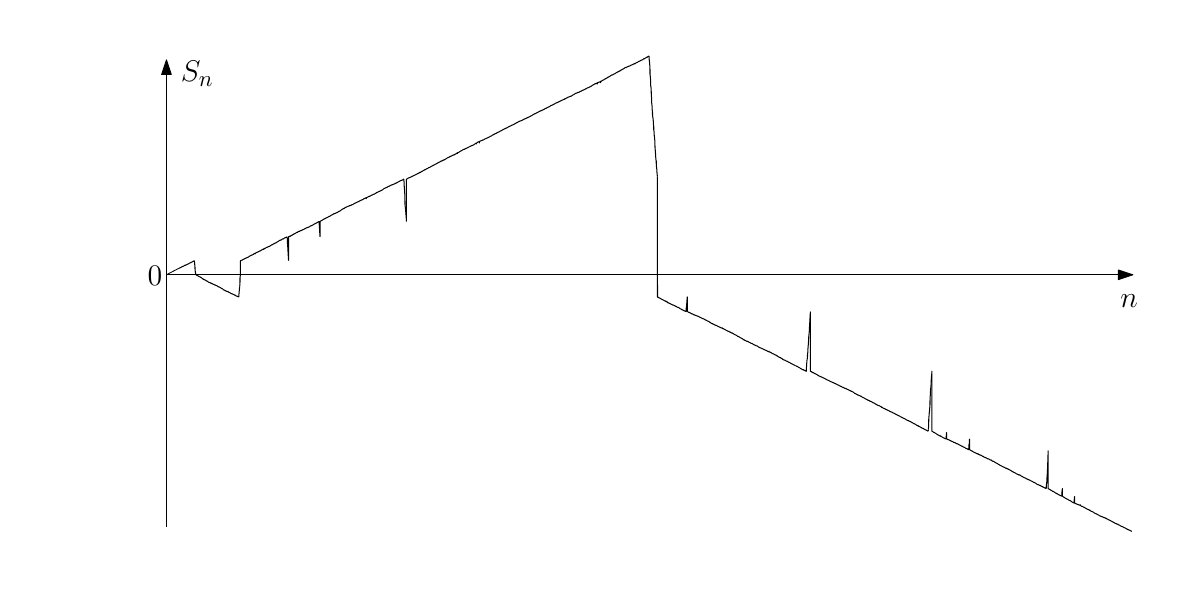}
}
\caption{\small First $4 \times 10^4$ steps of our model on the line with density $p=0.5$ of double marks among the Poisson points.
Note that the back-and-forth movement displayed here is usually not observed numerically on longer time scales (on the order of $10^7$ to $10^8$ steps for instance), where the plot strongly suggests ballisticity, in apparent conflict with our result.
This is consistent with the fact that the walk needs to build a sophisticated long-range bridge by leaving some points behind before it can jump across long gaps, hence taking extremely long time in doing so.
}
\label{fig:simul3}
\end{figure}

The heart of the proof goes through considerations of ergodic properties of the system.
It involves some reasoning by contradiction rather than a direct analysis of an apparently emerging multi-scale structure in the set of marks left behind, which enables the walk to jump across arbitrarily large gaps.
Yet, we show that the number of steps between consecutive jumps above the origin has infinite mean, and thus this model preserves part of the intricate mechanism that we expect to find on the strip.

\section{Model and Results}

Let $\X$ be a Poisson point process of intensity 1 on $\RR$, and $\Y$ be a $p$-thinning of~$\X$, for some $0<p<1$ fixed.
Informally, conditioned on~$\X$, the process $\Y$ is obtained by keeping each point of $\X$ independently with probability~$p$.
Viewing $\X$ and $\Y$ as random elements of the space $\mathcal N(\RR)$ of $\sigma$-finite, integer-valued measures on $\RR$ (corresponding to the counting measure of points), we also define the (multiple) point process
\[\omega=\X+\Y\]
that gives a weight 2 to each point in $\Y$ and 1 to each point in the rest of $\X$. We shall say that points in $\Y$ carry two \emph{marks}, while the other points in $\X$ carry one.
For a general background on point processes, the reader can see for instance~\cite{kallenberg-02}.

We will also view $\X$ and $\Y$ as random subsets and thus denote for instance ``$x\in\X$'' for ``$\X(\{x\})=1$''. For ease of reading, we frequently write $\omega(x)$ instead of $\omega(\{x\})$ to denote the number of marks at $x$ in the configuration~$\omega$. 

Let us denote by $\PP$ the underlying probability. We consider $\PP^1=\PP(\cdot\,|\,0\in\X)$ in the sense that, under $\PP^1$, $\X$ is sampled according to its Palm distribution (which amounts to adding a point at 0) and $\PP^2=\PP^1(\cdot\,|\,0\in\Y)$. Note that, under $\PP^1$, $\omega$ has two marks at 0 with probability $p$, and only one otherwise.

We now define the deterministic motion of a walk that visits the marks in $\omega$ in a greedy way.
Given $x\in\RR$, the walk starts at $S_0\in\X$ that is closest to $x$ (we may take $S_0=0$ under $\PP^1$), removes one mark from $S_0$ and moves to $S_1$, the nearest point in $\X\setminus\{S_0\}$. Then it removes one mark from $S_1$ and moves to $S_2$, the nearest point in $\X\setminus\{S_1\}$ still possessing at least one mark, and so on.

More formally, given a starting point $x\in\RR$, let $\omega_0=\omega$ and
\[S_0=\argmin\left\{ |y-x|\,:\,y\in\RR\text{ such that }y\ne x\text{ and }\omega(y)\ge1\right\}\]
be the closest point to $x$ in $\X$ (note that if $x=0$ then $S_0=0$ under $\PP^1$). Then we define the sequences $(\omega_n)_{n\ge0}$ and $(S_n)_{n\ge0}$ recursively by: for all $n\in\NN$, 
\[\omega_{n+1}=\omega_n-\delta_{S_n}\]
and
\[S_{n+1}=\argmin\left\{ |y-S_n|\,:\,y\in\RR\text{ such that }y\ne S_n\text{ and }\omega_n(y)\ge1\right\}.\]

When necessary, we will denote the sequence $(S_n)_{n\geq0}$ by $(S^{x}_n)_{n\geqslant 0}$ to make the initial $x$ explicit.
Otherwise, we assume that $S_0=0$. 
Notice that, for any $x\in\RR$, $\PP$-a.s.\ and $\PP^1$-a.s., the sequence $(S^{x}_n)_{n\geqslant 0}$ is well-defined, \textsl{i.e.}, there is always a unique $\argmin$ in the definition of $S_{n+1}$.

Let $\X_n = \{ x\in\RR:\omega_n(x)\geqslant 1\}$.
Since $(\omega_n)_n$ is decreasing, we have $\X_{n+1}\subseteq \X_n$ for all~$n$. 
Consider $\X_\oo = \bigcap_n \X_n$, the set of points left by the walk.

\begin{theorem}
\label{thm1}
$\PP^1$-a.s., $\X_\oo = \emptyset$.
\end{theorem}

For any subset $A$ of $\RR$, define also \[T_A=\inf\{n\geq1:S_n\in A\}\]
and, for any $x\in\RR$, $T_x=T_{\{x\}}$.

\begin{theorem}
\label{thm2}
$\EE^1[\Trm] = \infty$.
\end{theorem}

\section{Proofs}

Let us first introduce additional notation. When $0\in\X$, the points of $\X$ are labeled $(X_k)_{k\in\ZZ}$ in increasing order and so that $X_0=0$:
\[\cdots<X_{-2}<X_{-1}<X_0=0<X_1<X_2<\cdots\] 

Theorem~\ref{thm1} will follow from the proposition below.

\begin{proposition}\label{prop:visits_left}
$\PP^1(\Txm<\infty)=1$.
\end{proposition}

It will often be convenient to use an equivalent conditional formulation, such as
\begin{align}
\PP^1(\Txm<\infty|X_{-1})=1,\quad\PP^1\text{-a.s.}\label{eqn:condition1}
\end{align}
Since the above event depends only on $X_{-1}$ and $\omega_{|\RR_+}$, which are independent, this shows that Proposition~\ref{prop:visits_left} would still hold for any distribution of $X_{-1}$ on $(-\infty,0)$ that is absolutely continuous with respect to the original distribution of $X_{-1}$ or thus to Lebesgue measure.
And in the same way,
\begin{align}
\PP^1\left(\Txm<\infty\big|X_{-1}, X_1, \omega(0)\right)=1,\quad\PP^1\text{-a.s.,}\label{eqn:condition2}
\end{align}
which shows that the random variable $(-X_{-1}, X_1, \omega(0))$ may have any continuous distribution on $(0,+\infty)^{2}\times\{1,2\}$ (still being independent of $\omega_{|[X_2,+\infty)}(X_2+\cdot)$) without affecting the result. 

\begin{proof}
[Proof of Theorem~\ref{thm1}]
By Proposition~\ref{prop:visits_left}, $\Txm<\infty$ $\PP^1$-almost surely, \textsl{i.e.}, the walk visits almost surely~$X_{-1}$ at the first step or after a finite excursion to the right.

For $z\in\RR$, define the operator~$\theta_z$ which provides a configuration viewed from~$z$, \textsl{i.e.}, define $\theta_z\omega$ by $\theta_z\omega(A)=\omega(A-z)$ for any Borel subset $A\subseteq\RR$.
Also define the mirroring operator~$\sigma$ by $\sigma\omega(A)=\omega(-A)$.

Let
\[X'_1=\min\big(\X_{\Txm}\cap\RR_+\big).\]
By the time $\Txm$, the walk $S$ has stayed strictly on the right of $X_{-1}$, hence the restriction of $\theta_{X_{-1}}\omega_{\Txm}$ to $\RR_-$ is equal to the restriction of $\theta_{X_{-1}}\omega$ to $\RR_-$.
Moreover, $\Txm$ and $X_1,X_2,\dots,X_{\Txm}$ can be defined in terms of $\omega_{|\RR_+}$ and $X_{-1}$, and do not depend on $(\theta_{X_{-1}}\omega)_{|\RR_-}$.
Therefore, by the strong Markov property of $\omega$, the conditional distribution of $(\theta_{X_{-1}}\omega_{\Txm})_{|\RR_-}$ given $\Txm,X_1,X_2,\dots,X_{\Txm}$ is the same as the initial distribution of~$\omega_{|\RR_-}$. 

By~\eqref{eqn:condition1} applied to $\sigma\theta_{X_{-1}}\omega$, the walk must make a finite excursion to the left after $\Txm$ and eventually reach~$X'_1$.
At the time $T'$ of this visit to $X'_1$, the walk has stayed on the left of $X'_1$, but may have already visited $X'_1$ before $\Txm$.
As argued above, the conditional distribution of $(\theta_{X'_1}\omega_{T'})_{|\RR_+}$ is the same as the initial distribution of $\omega_{|\RR_+}$, \emph{except} possibly for the value of $\omega(X'_1)$ and a finite gap on the right of $X'_1$ (this happens if $X'_1$ is in $\Y$ and was visited before $X_{-1}$).
By~\eqref{eqn:condition2}, the proposition still applies to $\theta_{X'_1}\omega_{T'}$ and proves that the walks jumps over $0$ again after $T'$.
We then iterate the procedure, switching from one side to the other.
Since the previous arguments based on~\eqref{eqn:condition2} apply at each time, we conclude that every mark is eventually removed.
\end{proof}

We now need to prove Proposition~\ref{prop:visits_left}.

\begin{claim}
If $\PP^2(\Trm<\infty)=1$, then $\PP^1(\Txm<\infty)=1$.
\end{claim}

\begin{proof}
We relate the above probabilities via the following decomposition:
\begin{align*}
\PP^2(\Trm<\infty)
& =
\PP^2(S_1<0)+\PP^2(S_1>0,\, S \mbox{ visits } 0)
\\
& =
\frac{1}{2}+\EE^2 \left[\ \PP^2(S_1>0, S \mbox{ visits } 0\,|\,S_1)\right]
\\
& = 
\frac{1}{2}+
\int_{\RR_+} \PP^2(\, S \mbox{ visits } 0\, |\, S_1=z\, ) \, 
e^{-2z}
\dd z
\\ 
& =
\frac{1}{2}+
\int_{\RR_+} \PP^1(\, S^z \mbox{ visits } 0  \, |\,
S_1=z,\, \omega(0)=1\, )
\, e^{-2z} \dd z
\\
& =
\frac{1}{2}+
\int_{\RR_+} \PP^1(\, \Txm<\oo \, |\,
X_{-1}=-z,\, \omega(-z)=1\, )
\, e^{-2z} \dd z
\\
& =
\frac{1}{2}+
\int_{\RR_+} \PP^1(\, \Txm<\oo \, |\,
X_{-1}=-z\, )
\,  e^{-2z} \dd z
.
\end{align*}
If the left-hand side equals~$1$, so does the probability in the last line, for Lebesgue-a.e.\ positive $z$, and therefore the claim follows by integrating with respect to the distribution of $X_{-1}$ under $\PP^1$.
The third equality holds because the distribution of $S_1$ is Laplace of parameter 2 (indeed the law of $S_1$ is symmetric and $|S_1|$ is the minimum of two independent exponential random variables of parameter 1, \textsl{i.e.}~it is an exponential random variable of parameter 2).
The fourth equality is due to the property that $(S^0_{n+1},\omega_{n+1})_{n\geq 0}$ is equal to $(S^z_n,\omega'_n)_{n\geq0}$ for $\omega'=\omega-\delta_0$. The fifth equality follows from the invariance of~$\PP^1$ with respect to the random translation by ${X_1}$.
\end{proof}

By the above claim, proving the proposition reduces to proving
\begin{equation}
\label{eq:istaken}\tag{$\ast$}
\PP^2(\Trm=\infty) = 0
.
\end{equation}
In the sequel we define an observable $\R_y$ which in a way measures how difficult it is for the walk to return to $0$ once it arrives at a given position $y>0$.
The idea is that $\R_y$ is the size of the gap that the walk needs in order to ensure that~$0$ will be visited before it goes further right, see Lemma~\ref{lemma-gap}.

Let $y\in\X\cap \RR_+^*$.
If $\Trm<T_y$, we take $\R_y=0$.
Otherwise we have  $0<S_1,S_2,\dots,S_{T_y-1} < y = S_{T_y}$ and in this case we label the set $\X_{T_y}\cap[0,y]$ of remaining marks in decreasing order:
\[
0 = z_n < z_{n-1} < \cdots < z_{1} < z_0=y
,
\]
and define
\begin{align*}
\R_y &= \max\{y-z_j : y-z_{j} > 2 (y-z_{j-1}),\, j = 1,2,\dots,n \}
\\
&= y - \min\{z_j:z_{j-1}-z_j > y-z_{j-1},\, j = 1,2,\dots,n\}.
\end{align*}
Since $y-z_1$ is always in the above set, we have $0 < \R_y \leqslant y$.
Figure~\ref{fig:Ry} illustrates the definition of~$\R_y$.
\begin{lemma}
\label{lemma-gap}
If $X_{k+1}-X_k > \R_{X_k}$ and $0 \in \Y$, then $\Trm < T_{X_{k+1}}$ and in particular $\Trm < \infty$.
\end{lemma}
\begin{figure}[htb!]
\centering
\includegraphics[width=140mm]{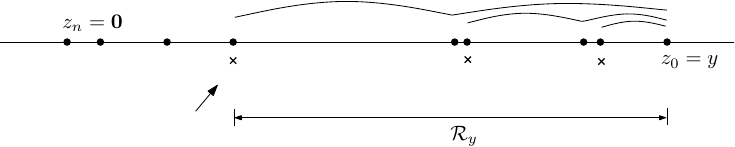}
\caption{Construction of~$\R_y$ from configuration $\X_{T_y}\cap[0,y]$.
There are $8$ points of $\X_{T_y}$ lying on $[0,y)$, so $n=8$ in this example.
Crosses indicate all the points which are farther from their next points than the next points are from~$y$, there are~$3$ such points in this example.
An arrow indicates the leftmost such point, and~$\R_y$ is its distance to~$y$.}
\label{fig:Ry}
\end{figure}

\begin{proof}
Assume that $A_k=[ X_{k+1}-X_k > \R_{X_k} ]$ occurs.
If $T_{X_k}=\infty$, then necessarily $\Trm<\infty$ and $T_{X_{k+1}} \geq T_{X_k} = \infty$.
We may thus assume $T_{X_k}<\infty$.
Denote as before the configuration $\X_{T_{X_k}}\cap[0,X_k]$ by
\[0=z_n<z_{n-1}<\cdots<z_1<z_0=X_k\]
assuming also $z_{-1}=X_{k+1}$.  
Note that exactly $1$ mark remains at $z_n,\ldots,z_1$ in the configuration $\omega_{T_{X_k}}$ at time $T_{X_k}$.
Denote by $l\geq 0$ the index such that $\R_{X_k}=X_k-z_l$.
Since, on $A_k$, $X_k-z_l=\R_{X_k}<X_{k+1}-X_k$, the walk must visit~$z_l$ (again) before it visits $X_{k+1}$ (and may, or may not remove $X_k$ completely if $X_k\in\Y$).
At this time $m'>m$ when $S_{m'}=z_l$, we have $\X_{m'}\cap(z_l,X_k)=\emptyset$, hence the nearest point to $S_{m'}$ on the right is at least as far as $X_k$.
From this instant on, the definition of $\R_{X_k}$ gives first that $S_{m'+1}=z_{l+1}$ (indeed, since $X_k-z_{l+1}<2(X_k-z_l)$, we have $z_l-z_{l+1}<X_k-z_l$), then that $S_{m'+2}=z_{l+2}$, and so on, until $S$ reaches $0$.
This concludes the proof of the lemma. 
\end{proof}

\begin{claim}
\label{cl:r0}
There exists $r_0 < \infty$ such that, $\PP^2$-a.s., $\R_x < r_0$ for infinitely many $x\in\X \cap \RR_+^*$.
\end{claim}

Let us see how the above claim implies~\eqref{eq:istaken}.

\begin{proof}[Proof of \eqref{eq:istaken}.] 
Let $\F_k=\sigma(X_1,\ldots,X_k)$.
For all $k\geq0$, $A_k:=[X_{k+1}-X_k>\R_{X_k}]\in\F_{k+1}$ and
\begin{align*}
\PP(A_k|\F_k)
	& \geq \PP(\R_{X_k}<r_0<X_{k+1}-X_k\,|\F_k)
=\I_{[\R_{X_k}<r_0]}\PP(r_0<X_{k+1}-X_k)\\
	&=\I_{[\R_{X_k}<r_0]}e^{-r_0}.
\end{align*}
Hence Claim~\ref{cl:r0} gives that $\sum_{k\geq0}\PP(A_k|\F_k)=\infty$ almost surely.
Finally~(\ref{eq:istaken}) follows from the conditional Borel-Cantelli Lemma~\cite[Thm~5.3.2]{durrett-10} and Lemma~\ref{lemma-gap}.
\end{proof}

To complete the proof of Theorem~\ref{thm1}, we prove Claim~\ref{cl:r0} using ergodic arguments.

\begin{proof}
[Proof of Claim~\ref{cl:r0}]
Let $\Z \subseteq \Y$ be the set of points $x\in\Y$ such that, for all $n\geq1$, $S^x_n > x$. 
The event in~(\ref{eq:istaken}) is equivalent to $[0 \in \Z]$.
Let us write
\[\delta = \PP^2(0 \in \Z).\]

If $0\notin\Z$, then note that for all $y\geq0$ large enough, $T_0<T_y$ hence by definition $\R_y=0$, and the conclusion of the claim holds for any $r_0>0$. For that reason, it will be sufficient to prove the claim under the assumption $\delta>0$ and on the event $[0\in\Z]$.  

Thus, let us assume $\delta>0$. 
By definition, $\Z$ is a function of $\omega$ that is translation covariant (meaning that $\Z(\theta_x\omega)=\theta_x\Z(\omega)$ for any translation $\theta_x$). 
Therefore, under $\PP$, $\Z$ is a stationary and ergodic point process in $\RR$. The density of $\Z$ is  $\delta p$ (and in particular $\Z$ is non-empty since $\delta>0$).

For every $y\in\Z$, let us define a quantity $R_y$ similar to $\R_y$ but that will be stationary in space.
Label the points of $\Z\cap(-\infty,y]$ in decreasing order \[\cdots<z_2<z_1<z_0=y,\] and take
\[R_y=\sup\{y-z_j\,:\,y-z_j>2(y-z_{j-1}),\, j\geq 1\}.\]

Most importantly, these random variables are finite: almost surely, $R_y<\infty$ for all $y\in\Z$.
This follows from the fact that the ergodic set $\Z$ has a density $\delta p>0$ almost surely.
Indeed, this density implies that almost surely
\[-z_j\,\equivalent{j\to\infty}\,\frac1{\delta p}j\]
hence, for any $y\in\Z$, $y-z_j<2(y-z_{j-1})$ for large $j$ and thus $R_y<\infty$.  

In addition, as a function of $\Z$, $(R_y)_{y\in\Z}$ is translation covariant, and thus the marked process $\big(\Z,(R_y)_{y\in\Z}\big)$ is stationary and ergodic.
In particular, if we choose $r_0$ large enough so that $\PP^2(R_y<r_0\,|\,0\in\Z)>0$ then, $\PP^2$-almost surely, $R_y<r_0$ for infinitely many points $y\in\Z \cap \RR_+$.

Let us finally argue that on $[0\in\Z]$, for all $y\in\Z\cap\RR_+$, $\R_y \leq R_y$.
Given the previous conclusion, this will readily imply the claim. 

First, we can waive the condition $j\leq n$ on the definition of $\R_y$: 
\[\R_y\leq \sup\{y-z_j\,:\,y-z_j>2(y-z_{j-1}), j\geq 1\}.\]
The difference with $R_y$ then lies in the set of points $z_j$. Let us justify that
\begin{itemize}
\item the set of points $z_j$ considered for $R_y$ is a subset of those considered for $\R_y$;
\item adding points in the definition decreases the supremum. 
\end{itemize}
The first property follows from the fact that points in $\Z\cap\RR_+$ are never fully removed.
Indeed, on the first visit to a point $z$ in this set, the gap on its left is greater than it initially was in $\X$ and the subset on its right is untouched, hence the definition of $\Z$ ensures that the walk then stays on the right and leaves the second mark at $z$ intact.

The second property comes from the following equivalent ``percolative'' definition of $R_y$, which for simplicity we write for $R_0$:
\[
R_0=\min\bigg\{\ell>0 :
[-\ell,0]\cup\bigcup_{x\in\Z\cap(-\infty,0]}[2x,x]=\RR_-\bigg\}=- \inf
\bigg(\RR_-\setminus\bigcup_{x\in\Z\cap(-\infty,0]}[2x,x]\bigg)
.
\]
The equivalence between the definitions follows from the observation that
$X_j<2X_{j-1}$ if and only if $[2X_j,X_j]\cap[2X_{j-1},X_{j-1}]=\emptyset$ and
in that case
$X_j=\inf([2X_j,X_{j-1}]\setminus([2X_j,X_j]\cap[2X_{j-1},X_{j-1}]))$. 

This concludes the proof of the comparison $\R_y<R_y$, and thus of the claim. 
\end{proof}

\begin{proof}
[Proof of Theorem~\ref{thm2}]

We introduce a ``renewal structure'' on $\X\cap\RR_+$. 
Let us momentarily only consider, under $\PP^1$, the subset of positive marks plus exactly one at zero:
\[\omega^+=\omega_{|(0,+\infty)}+\delta_0.\]
The following definitions refer to this distribution of marks. 

For any $x\in\X\cap\RR_+$, let us denote its removal time (or last hitting time) by
\[\mathcal T^R_x(\omega^+)=\min\{n\geq0\,|\,\omega^+_{n+1}(x)=0\},\]
and consider the maximum displacement $\ell(x)$ of $S$ to the right of $x$ before that time, 
in other words
\[\ell(x)=\max_{0\leq n\leq\mathcal T^R_x(\omega^+)}S_n(\omega^+)-x.\]
Here we let a priori $\ell(x)=\infty$ if $\mathcal T^R_x(\omega^+)=\infty$ (which by Theorem~\ref{thm1} actually does almost surely not happen, but we do not need Theorem~\ref{thm1} for Theorem~\ref{thm2}).
In particular $\ell(x)=0$ if $\omega(x)=1$ (the first visit at $x$ already removes this point).
Let us also define $\xi(x)$ to be the next gap after $x+\ell(x)$:
\[\xi(x)=\big(x+\ell(x)\big)'-\big(x+\ell(x)\big)\]
where, for $z\in\X$, $z'$ denotes the next mark on the right of $z$.
Finally let
\[L(x)=\ell(x)+\xi(x), \]
so that, as long as $n\leq \mathcal T^R_x(\omega^+)+1$, $S_n(\omega^+)$ is fully determined by the restriction of $\omega^+$ to $[0,x+L(x)]$, and does actually not depend on whether $\omega^+(x+L(x))=1$ or $2$. 

Note that, since points carry at most two marks, for any $x\in\X\cap\RR_+$, at the hitting time of $x+L(x)$, there is no mark left in $[x,x+L(x)[$.
Indeed, $x$ has been cleared at that time and every piece of the interval $[x,x+\ell(x)[$ has been crossed (at least) twice by the walk, thus removing every mark in it.
Furthermore, even if one mark may possibly have been left at $x+\ell(x)$ at time $\mathcal T^R_x(\omega_+)$, it has to have been removed before hitting the point $x+L(x)$, lying on its right. 

For that reason, if we define recursively the lengths $L_0=L(0)=X_1$ and for all $k\geq0$
\[L_{k+1}=L(L_0+\cdots+L_k),\]
(as long as $L_k<\infty$) we get by induction that for all $k\geq1$, at the hitting time of $L_0+\cdots+L_k$, there is no mark anymore on the left of that point. And therefore, using the previous remark, the random variables~$(L_k)_{k\geq1}$ are i.i.d.

Assume now that $\omega(0)=2$ and $S_1=X_1$. Thus, after the first step, one mark is left at~0, contrary to what happens on $\omega^+$. The first difference with the process on $\omega^+$ will occur when the walk is closer to~$0$ than to the next mark on its right, causing a first return to~$0$. This may happen only in one of three types of situations: (introducing $\ell_{k+1}=\ell(L_0+\cdots+L_k)$ and $\xi_{k+1}=\xi(L_0+\cdots+L_k)$)
\begin{itemize}
\item either on the first hitting time of $L_0+\cdots+L_k$ for some $k$, if $\zeta_{k+1}>L_0+\cdots+L_k$ where $\zeta_{k+1}$ is the first gap after $L_0+\cdots+L_k$ (note that this implies $L_{k+1}>L_0+\cdots+L_k$);
\item or after the return time (if any) to $L_0+\cdots+L_k$ for some $k$, if either the next mark lies at $L_0+\cdots+L_{k+1}$ and $L_{k+1}>L_0+\cdots+L_k$, or if one mark is left at $L_0+\cdots+L_k+\ell_{k+1}$ and $\ell_{k+1}>L_0+\cdots+L_k$ (which implies $L_{k+1}>L_0+\cdots+L_k$);
\item or after the next step after the hitting time of $L_0+\cdots+L_k$ for some $k$, if a mark at $L_0+\cdots+L_k+\ell_{k+1}$ is still present, and if $\xi_{k+1}>L_0+\cdots+L_k+\ell_{k+1}$ (this also implies $L_{k+1}>L_0+\cdots+L_k$).
\end{itemize}
Thus in each case we get that a return to $0$ implies $L_{k+1}>L_0+\cdots+L_k$ for some $k\geq0$. 

Assume, by contradiction, that $\EE[L_1]<\infty$ (hence $L_1<\infty$ a.s.).
Then it follows from the law of large numbers that
\[\PP(\forall k\geq1,\ L_{k+1}<L_1+\cdots+L_k)>0.\]
Therefore, $\PP(\mathcal T_0=\infty)>0$, and by symmetry $\PP(\mathcal T_0=\infty, S_1=X_1)>0$. But this would imply that $\PP(L_1=\infty)>0$ since the environment $\omega^+$ on the right of $X_1$ has same law as the one on the right of $0$ under $\PP$. And thus $\EE[L_1]=\infty$, a contradiction. Therefore, $\EE[L_1]=\infty$. 

To conclude the proof of Theorem~\ref{thm2}, let us argue that $\EE[L_1]=\infty$ implies $\EE^1[\Trm]=\infty$. The event $[\omega(0)=2,S_1=X_1]$ has positive probability $\frac p2$, is independent of $L_1$ and, on this event, $\Trm$ is at least equal to  the number of marks between $X_1$ and $X_1+L_1$. Thus, if we denote $\xi_k=X_{k+1}-X_k$ for all $k\geq0$, these variables are i.i.d.~exponential r.v.~with parameter 1, and on $[\omega(0)=2,S_1=X_1]$,
\[L_1\leq\sum_{k=1}^{\Trm} \xi_k.\]
Hence for any $k>0$,
\begin{align*}
\PP^1(\Trm>k\,|\,\omega(0)=2,S_1=X_1)
& \geq \PP(\xi_1+\cdots+\xi_k< 2k,\,L_1>2k) \\
& \geq \PP(L_1>2k)-\PP(\xi_1+\cdots+\xi_k> 2k). 
\end{align*}
Since $E[\xi_i]=1<2$, the last probability is seen to decrease exponentially (by large deviation principle, or Chernoff's inequality), and since $\EE[L_1]=\infty$ the second to last probability is seen to sum to $+\infty$ with respect to $k\in\NN$. Therefore, summing on both sides with respect to $k\in\NN$ gives
\[\EE^1[\Trm\,|\,\omega(0)=2,S_1=X_1]=\infty,\]
hence the result.
\end{proof}

\section*{Acknowledgments}

L.T. thanks Claude Loverdo for telling him about the papers~\cite{santos-et-al-07,boyer-08}. We thank the MSRI in Berkeley, the DMA in ENS-Paris, Universit\'e Paris XIII, and the MRI at Oberwolfach for their hospitality.
V.S. thanks FIM of ETH Zurich for hospitality and financial support. 
L.T.R. thanks the Brazil-France agreement.
This work was supported by ESF project RGLIS and ANR project MEMEMO2 2010 BLAN 0125.

\renewcommand{\baselinestretch}{1}
\parskip 0pt
\small
\bibliographystyle{bib/rollaalphasiam}
\bibliography{bib/leo}

\end{document}